\documentclass[12pt]{amsart}
\usepackage{mathrsfs}
\usepackage{}
\usepackage{amsmath}
\usepackage{amsfonts}
\usepackage{amssymb}
\usepackage[all,cmtip]{xy}           
\usepackage{shuffle}
\usepackage{caption}
\usepackage{bbding}
\usepackage{txfonts}
\usepackage[shortlabels]{enumitem}
\usepackage{ifpdf}
\ifpdf
\usepackage[colorlinks,final,backref=page,hyperindex]{hyperref}
\else
\usepackage[colorlinks,final,backref=page,hyperindex,hypertex]{hyperref}
\fi
\usepackage{tikz}
\usepackage[active]{srcltx}

\makeatletter

\newtheorem{theorem}{Theorem}[section]
\newtheorem{prop}[theorem]{Proposition}
\newtheorem{lemma}[theorem]{Lemma}
\newtheorem{coro}[theorem]{Corollary}
\newtheorem{prop-def}{Proposition-Definition}[section]

\theoremstyle{definition}
\newtheorem{defn}[theorem]{Definition}

\newtheorem{remark}[theorem]{Remark}
\newtheorem{exam}[theorem]{Example}

\newcommand{\nc}{\newcommand}

\newcommand {\emptycomment}[1]{}

\nc{\delete}[1]{{}}
\nc{\mmargin}[1]{}

\nc{\mlabel}[1]{\label{#1}}  
\nc{\mcite}[1]{\cite{#1}}  
\nc{\mref}[1]{\ref{#1}}  
\nc{\meqref}[1]{\eqref{#1}}  
\nc{\mbibitem}[1]{\bibitem{#1}} 

\newcommand{\bk}{{\mathbf{k}}}

\nc{\vep}{\varepsilon}
\nc{\bin}[2]{ (_{\stackrel{\scs{#1}}{\scs{#2}}})}  
\nc{\binc}[2]{(\!\! \begin{array}{c} \scs{#1}\\
		\scs{#2} \end{array}\!\!)}  
\nc{\bincc}[2]{  ( {\scs{#1} \atop
		\vspace{-1cm}\scs{#2}} )}  
\nc{\oline}[1]{\overline{#1}}
\nc{\mapm}[1]{\lfloor\!|{#1}|\!\rfloor}
\nc{\bs}{\bar{S}}
\nc{\cast}{{\,\mbox{\raisebox{.8pt}{$\scriptstyle \circledast$}}\,}}
\nc{\la}{\longrightarrow}
\nc{\ot}{\otimes}
\nc{\rar}{\rightarrow}
\nc{\dar}{\downarrow}
\nc{\dap}[1]{\downarrow \rlap{$\scriptstyle{#1}$}}
\nc{\defeq}{\stackrel{\rm def}{=}}
\nc{\dis}[1]{\displaystyle{#1}}
\nc{\dotcup}{\ \displaystyle{\bigcup^\bullet}\ }
\nc{\hcm}{\ \hat{,}\ }
\nc{\hts}{\hat{\otimes}}
\nc{\hcirc}{\hat{\circ}}
\nc{\lleft}{[}
\nc{\lright}{]}
\nc{\curlyl}{\left \{ \begin{array}{c} {} \\ {} \end{array}
	\right .  \!\!\!\!\!\!\!}
\nc{\curlyr}{ \!\!\!\!\!\!\!
	\left . \begin{array}{c} {} \\ {} \end{array}
	\right \} }
\nc{\longmid}{\left | \begin{array}{c} {} \\ {} \end{array}
	\right . \!\!\!\!\!\!\!}
\nc{\ora}[1]{\stackrel{#1}{\rar}}
\nc{\ola}[1]{\stackrel{#1}{\la}}
\nc{\scs}[1]{\scriptstyle{#1}} \nc{\mrm}[1]{{\rm #1}}
\nc{\dirlim}{\displaystyle{\lim_{\longrightarrow}}\,}
\nc{\invlim}{\displaystyle{\lim_{\longleftarrow}}\,}
\nc{\dislim}[1]{\displaystyle{\lim_{#1}}} \nc{\colim}{\mrm{colim}}
\nc{\mvp}{\vspace{0.3cm}} \nc{\tk}{^{(k)}} \nc{\tp}{^\prime}
\nc{\ttp}{^{\prime\prime}} \nc{\svp}{\vspace{2cm}}
\nc{\vp}{\vspace{8cm}}
\nc{\modg}[1]{\!<\!\!{#1}\!\!>}
\nc{\intg}[1]{F_C(#1)}
\nc{\lmodg}{\!<\!\!}
\nc{\rmodg}{\!\!>\!}
\nc{\cpi}{\widehat{\Pi}}
\nc{\labs}{\mid\!}
\nc{\rabs}{\!\mid}
\nc{\btr}{\blacktriangleright}

\nc{\ad}{\mrm{ad}}
\nc{\rRB}{\mathsf{rRB}}
\nc{\cocrRB}{\mathsf{cocrRB}}
\nc{\PH}{\mathsf{PH}}
\nc{\cocPH}{\mathsf{cocPH}}
\nc{\ann}{\mrm{ann}}
\nc{\Ad}{\mrm{Coad}}
\nc{\Aut}{\mrm{Aut}}
\nc{\Der}{\mrm{Der}}
\nc{\Sym}{\mrm{Sym}}
\nc{\br}{\mrm{bre}}
\nc{\can}{\mrm{can}}
\nc{\Cont}{\mrm{Cont}}
\nc{\rchar}{\mrm{char}}
\nc{\cok}{\mrm{coker}}
\nc{\de}{\mrm{dep}}
\nc{\dtf}{{R-{\rm tf}}}
\nc{\dtor}{{R-{\rm tor}}}

\nc{\Dif}{\mrm{Diff}}
\nc{\Div}{\mrm{Div}}
\nc{\End}{\mrm{End}}
\nc{\Ext}{\mrm{Ext}}
\nc{\Fil}{\mrm{Fil}}
\nc{\Fr}{\mrm{Fr}}
\nc{\Frob}{\mrm{Frob}}
\nc{\Gal}{\mrm{Gal}}
\nc{\GL}{\mrm{GL}}
\nc{\Gr}{\mrm{Gr}}
\nc{\Hom}{\mrm{Hom}}
\nc{\Hoch}{\mrm{Hoch}}
\nc{\hsr}{\mrm{H}}
\nc{\hpol}{\mrm{HP}}
\nc{\id}{\mrm{id}}
\nc{\im}{\mrm{im}}
\nc{\inv}{\mrm{inv}}
\nc{\Id}{\mrm{Id}}
\nc{\ID}{\mrm{ID}}
\nc{\Irr}{\mrm{Irr}}
\nc{\incl}{\mrm{incl}}
\nc{\length}{\mrm{length}}
\nc{\NLSW}{\mrm{NLSW}}
\nc{\Lie}{\mrm{Lie}}
\nc{\mchar}{\rm char}
\nc{\mpart}{\mrm{part}}
\nc{\ql}{{\QQ_\ell}}
\nc{\qp}{{\QQ_p}}
\nc{\rank}{\mrm{rank}}
\nc{\rcot}{\mrm{cot}}
\nc{\rdef}{\mrm{def}}
\nc{\rdiv}{{\rm div}}
\nc{\rtf}{{\rm tf}}
\nc{\rtor}{{\rm tor}}
\nc{\res}{\mrm{res}}
\nc{\SL}{\mrm{SL}}
\nc{\Spec}{\mrm{Spec}}
\nc{\tor}{\mrm{tor}}
\nc{\Tr}{\mrm{Tr}}
\nc{\tr}{\mrm{tr}}
\nc{\wt}{\mrm{wt}}

\nc{\bfk}{{\bf k}}
\nc{\bfone}{{\bf 1}}
\nc{\bfzero}{{\bf 0}}
\nc{\detail}{\marginpar{\bf More detail}
	\noindent{\bf Need more detail!}
	\svp}
\nc{\gap}{\marginpar{\bf Incomplete}\noindent{\bf Incomplete!!}
	\svp}
\nc{\FMod}{\mathbf{FMod}}
\nc{\Int}{\mathbf{Int}}
\nc{\Mon}{\mathbf{Mon}}
\nc{\remarks}{\noindent{\bf Remarks: }}
\nc{\Rep}{\mathbf{Rep}}
\nc{\Rings}{\mathbf{Rings}}
\nc{\Sets}{\mathbf{Sets}}
\nc{\Diff}{\mathbf{Diff}}
\nc{\Inte}{\mathbf{Inte}}
\nc{\U}{\mathrm{U}}
\newcommand{\YD}{{\mathcal{Y}\mathcal{D}}}

\nc{\BA}{{\mathbb A}}   \nc{\CC}{{\mathbb C}}
\nc{\DD}{{\mathbb D}}   \nc{\EE}{{\mathbb E}}
\nc{\FF}{{\mathbb F}}   \nc{\GG}{{\mathbb G}}
\nc{\HH}{{\mathbb H}}   \nc{\LL}{{\mathbb L}}
\nc{\NN}{{\mathbb N}}   \nc{\PP}{{\mathbb P}}
\nc{\QQ}{{\mathbb Q}}   \nc{\RR}{{\mathbb R}}
\nc{\TT}{{\mathbb T}}   \nc{\VV}{{\mathbb V}}
\nc{\ZZ}{{\mathbb Z}}   \nc{\TP}{\widetilde{P}}

\nc{\cala}{{\mathcal A}}    \nc{\calc}{{\mathcal C}}
\nc{\cald}{\mathcal{D}}     \nc{\cale}{{\mathcal E}}
\nc{\calf}{{\mathcal F}}    \nc{\calg}{{\mathcal G}}
\nc{\calh}{{\mathcal H}}    \nc{\cali}{{\mathcal I}}
\nc{\call}{{\mathcal L}}    \nc{\calm}{{\mathcal M}}
\nc{\caln}{{\mathcal N}}    \nc{\calo}{{\mathcal O}}
\nc{\calp}{{\mathcal P}}    \nc{\calr}{{\mathcal R}}

\nc{\cals}{{\mathcal S}}    \nc{\calt}{{\Omega}}
\nc{\calv}{{\mathcal V}}    \nc{\calw}{{\mathcal W}}
\nc{\calx}{{\mathcal X}}

\nc{\fraka}{{\mathfrak a}}
\nc{\frakb}{\mathfrak{b}}
\nc{\frakg}{{\frak g}}
\nc{\frakl}{{\frak l}}
\nc{\fraks}{{\frak s}}
\nc{\frakB}{{\frak B}}
\nc{\frakm}{{\frak m}}
\nc{\frakM}{{\frak M}}
\nc{\frakp}{{\frak p}}
\nc{\frakW}{{\frak W}}
\nc{\frakX}{{\frak X}}
\nc{\frakS}{{\frak S}}
\nc{\frakA}{{\frak A}}
\nc{\frakx}{{\frakx}}

\nc{\ynr}[1]{\textcolor{orange}{\underline{Yunnan:}#1 }}

\nc{\lir}[1]{\textcolor{red}{\underline{Li:}#1 }}

\begin{document}

\title[Matched pairs and Yang-Baxter operators]{Matched pairs and Yang-Baxter operators}

\author{Yunnan Li}
\address{School of Mathematics and Information Science, Guangzhou University,
Guangzhou 510006, China}
\email{ynli@gzhu.edu.cn}

\begin{abstract}
Recently, Ferri and Sciandra introduced two equivalent algebraic structures, matched pair of actions on an arbitrary Hopf algebra and Yetter-Drinfeld brace. In fact, they equivalently produce braiding operators on Hopf algebras satisfying the braid equation, thus generalize the construction of Yang-Baxter operators by Lu, Yan and Zhu from braiding operators on groups, and also by Angiono, Galindo and Vendramin from cocommutative Hopf braces. In this paper, we provide equivalence conditions for such kind of Yang-Baxter operators to be involutive. Particularly, we give a positive answer for an open problem raised by Ferri and Sciandra, namely, a matched pair of actions on a Hopf algebra $H$ induces an involutive Yang-Baxter operator if and only if its intrinsic Hopf algebra $H_\rightharpoonup$ in the category of Yetter-Drinfeld modules over $H$ is braided commutative.
Also, we show that the double cross product $H\bowtie H$ is a Hopf algebra with a projection and $H_\rightharpoonup$ serves as its subalgebra of coinvariants. As an illustration, we use a simplified characterization to classify matched pairs of actions on the 8-dimensional non-semisimple Hopf algebra $A_{C_2\times C_2}$ and analyze the associated Yang-Baxter operators to find that they are all involutive.
\end{abstract}

\keywords{matched pair, Yang-Baxter operator, Yetter-Drinfeld module, braided Hopf algebra
\\
\qquad 2020 Mathematics Subject Classification. 16T05, 16T25, 18M15}

\maketitle

\tableofcontents

\allowdisplaybreaks

\section{Introduction}

The Yang-Baxter equation is one of the fundamental equations in mathematical physics,
initially appearing in both quantum and statistical mechanics, and introducing the theory of quantum group. How to systematically construct and classify its solutions is an extremely challenging open problem.

Given a vector space $V$, let $R: V \otimes  V\to  V \otimes V$ be a linear automorphism satisfying the Yang-Baxter equation, namely
$$R_{12}R_{13}R_{23}=R_{23}R_{13}R_{12}.$$
Denote $\tau$ the usual flip map, i.e. $\tau(u\otimes v)=v\otimes u$. Then $r=\tau R$ is a solution of the braid equation
$$(r \otimes \id)(\id \otimes r)(r \otimes \id) = (\id \otimes r)(r \otimes \id)(\id \otimes r),$$
and called a {\bf Yang-Baxter operator}. It gives rise to a representation of the braid group $B_n$ on the tensor space $V^{\otimes n}$ for any $n$. In particular, if $r$ is involutive, i.e. $r^2=\id^{\otimes 2}$, then it becomes a representation of the symmetric group $S_n$ on $V^{\otimes n}$.

In 1992, Drinfeld proposed in \cite{Dr} the study of  set-theoretic solutions of the Yang-Baxter equation. In the same decade, Gateva-Ivanova and Van den Bergh \cite{GVan}, Etingof, Schedler and Soloviev \cite{Etingof}, as well as Lu, Yan and Zhu \cite{LYZ}, gave their seminal work independently. Specifically, Lu, Yan and Zhu \cite{LYZ} introduced a general construction method for set-theoretic solutions to the Yang-Baxter equation using braiding operators on groups. Afterwards, Takeuchi noticed in \cite{Ta} that the LYZ construction of braided group is equivalent to certain matched pair of groups with compatible actions.

So far the study of set-theoretic solutions of the Yang-Baxter equation has got explosive development. An effective and fruitful approach
is to identify and study the underlying algebraic structures. More precisely,
one can focus on e.g. the (semi)group and ring theoretical aspects that occur, and study them for specific classes of set-theoretic solutions of the Yang-Baxter equation. For example, the milestone research \cite{Ru} of Rump introduced braces as a tool to classify all non-degenerate involutive set-theoretic solutions, and one decade later Guarnieri and Vendramin in~\cite{GV} introduced the notion of skew brace to produce non-degenerate solutions in general, and pointed out that skew left brace is equivalent to LYZ's braided group. Recently, one more equivalent notion, namely post-(Lie) group, was also introduced in \cite{BGST} as the integration of post-Lie algebras.

On the other hand, matched pair of Hopf algebras was introduced by Takeuchi in \cite{Ta1}, and presented by Majid with the modern definition; see~\cite{Ma1,Ma}.
Generalizing the above framework from groups to Hopf algebras, Angiono, Galindo, and Vendramin in \cite{AGV} introduced the notion of Hopf brace, and demonstrated that cocommutative Hopf braces are equivalent to matched pairs on cocommutative Hopf algebras with compatible actions, and produce Yang-Baxter operators.
Recently, the construction of Yang-Baxter operators in \cite{AGV} from Hopf braces was reformulated in~\cite{LST} from the perspective of cocommutative post-Hopf algebras, with the universal enveloping algebra of a post-Lie algebra~\cite{MW,ELM} as a basic model.

Also, Guccione, Guccione and Vendramin in \cite{GGV} extended the ideas in \cite{AGV} to the general setting of non-degenerate solutions in symmetric tensor categories, and proposed braiding operators on cocommutative Hopf algebras. This approach was further generalized to arbitrary Hopf algebras by Guccione, Guccione and Valqui in \cite{GGV0}, where they established a more general framework, so-called weak braiding operator and Hopf $q$-brace, relating to left non-degenerate solutions of the braid equation.

As another approach to removing the cocommutativity hypothesis in \cite{AGV}, Ferri and Sciandra \cite{FS} recently introduced the notion of matched pair of actions as matched pair on a Hopf algebra with compatible actions, and innovatively discovered its equivalent structure, so-called Yetter-Drinfeld brace. Roughly speaking,
Yetter-Drinfeld brace replaces the first ordinary Hopf algebra in a Hopf brace by a Hopf algebra in the Yetter-Drinfeld module category over the second one.
Relating to Majid's transmutation theory, the authors in \cite{FS} also pointed out that coquasitriangular structures produce matched pairs of actions and provided concrete examples to illustrate this result. During our study about the work \cite{FS}, we accidentally found that  Guccione, Guccione and Valqui have recognized the equivalence between matched pairs of actions on Hopf algebras and braiding operators on Hopf algebras satisfying the braid relation, and pointed it out in their updated version~\cite{GGV1}.
In addition, Sciandra further obtained one more equivalent structure, so-called Yetter-Drinfeld post-Hopf algebra, in \cite{Sc}, generalizing the equivalence between cocommutative Hopf braces and cocommutative post-Hopf algebras.

In summary, the aforementioned algebraic structures have the following connection.
\begin{itemize}
\item
According to \cite{AGV,GGV,LST}, the following four categories are all isomorphic:
\begin{enumerate}[(1)]
\item
matched pairs of actions on cocommutative Hopf algebras
\item
braiding operators on cocommutative Hopf algebras
\item
cocommutative Hopf braces
\item
cocommutative post-Hopf algebras
\end{enumerate}

\smallskip\item
In \cite{FS,GGV1,Sc} the above isomorphisms were generalized to four larger categories:
\begin{enumerate}[(1)]
\item
matched pairs of actions on Hopf algebras
\item
braiding operators on Hopf algebras
\item
Yetter-Drinfeld braces
\item
Yetter-Drinfeld post-Hopf algebras
\end{enumerate}
\end{itemize}
In order to simplify the terminology, we only focus on matched pair of actions, braiding operator and Yetter-Drinfeld brace here, while the post-Hopf counterpart will not be discussed.

In this paper, we first revisit the equivalence between matched pairs of actions on an arbitrary Hopf algebra and braiding operators on it which satisfy the braid relation, so it is meaningful to study matched pairs of actions as a rich source for Yang-Baxter operators.
Then as the main result, we provide several equivalence conditions for such kind of braiding operators to be involutive. Particularly, we solve a problem left in~\cite{FS} by confirming that if a matched pair of actions on a Hopf algebra $H$ induces an involutive braiding operator, then its intrinsic Hopf algebra in the Yetter-Drinfeld $H$-module category is braided commutative, and vice versa. Namely, the role of braces inside the family of skew braces has an extension in the context of Yetter-Drinfeld braces.

The paper is organized as follows. In Section~\ref{sec:mp}, we recall the main notions, namely matched pair of actions on a Hopf algebra and Yetter-Drinfeld brace.
In Section~\ref{sec:yb}, we first reprove that a matched pair of actions on a Hopf algebra $H$ is equivalent to a braiding operator $r$ on $H$ satisfying the braid relation, especially a Yang-Baxter operator~(Theorem~\ref{thm:yb}) and give some concise formulas involving $r$ and the antipode of $H$~(Proposition~\ref{prop:ybo}).
Next we provide several equivalence conditions for the Yang-Baxter operator $r$ to be involutive~(Theorem~\ref{thm:involutive}), and then justify them by given the explicit formulas for the adjoint actions of the intrinsic braided Hopf algebra  $H_\rightharpoonup$ from a matched pair of actions on $H$~(Proposition~\ref{prop:adjoint}). Also, we find a Hopf algebra isomorphism between the double cross product $H\bowtie H$ and the bosonization $H_\rightharpoonup\# H$~(Theorem~\ref{thm:dcp_bos}).
In Section~\ref{sec:construct}, we use a simplified characterization of matched pairs of actions on Hopf algebras~(Theorem~\ref{thm:construct-mp}) to classify matched pairs of actions on an $8$-dimensional non-semisimple Hopf algebra $A_{C_2\times C_2}$~(Example~\ref{ex:A22}), and then analyze their associated Yang-Baxter operators (Theorem~\ref{thm:A22}) as an illustration.

\vspace{2mm}

\noindent
{\bf Convention.}
In this paper, we fix an algebraically closed ground field $\bk$ of characteristic 0.
All the objects under discussion, including vector spaces, algebras and tensor products, are taken over $\bk$ by default.

For any unital algebra $(A,m,u)$ with multiplication $m$ and the unit map $u:\bk\to A$, let $m^{(0)}=\id$ and for $n\geq1$ we write
$$m^{(n)}=(m\otimes \id^{\otimes(n-1)})\cdots (m\otimes\id)m.$$

For any coalgebra $(C,\Delta,\vep)$, we compress the Sweedler notation of the comultiplication $\Delta$ as $$\Delta(x)=x_1\otimes x_2$$ for simplicity.
Furthermore, let $\Delta^{(0)}=\id$ and for $n\geq1$ we write
$$\Delta^{(n)}(x)=(\Delta\otimes\id^{\otimes (n-1)})\cdots(\Delta\otimes\id)\Delta(x)=x_1\otimes\cdots\otimes x_{n+1}.$$

By convention, a Hopf algebra is denoted by $H=(H,\cdot\,,1,\Delta,\vep,S)$.
Denote by $G(H)$ the set of group-like elements in $H$, which is a group. Denote by $P_{g,h}(H)$ the subspace of $(g,h)$-primitive elements in $H$ for $g,h\in G(H)$.
For other basic notions of Hopf algebras, we follow the textbooks~\mcite{Mon}.

\section{Matched pairs of actions and Yetter-Drinfeld braces}\label{sec:mp}

First we recall the notion of matched pair of Hopf algebras
in Majid's book~\cite{Ma}.
\begin{defn}[{\cite[\S~7.2]{Ma}}]
A {\bf matched pair of Hopf algebras} is a quadruple $(H,K,\rightharpoonup,\leftharpoonup)$, where $H$ and $K$
are Hopf algebras, $\rightharpoonup$ is a left $H$-module coalgebra action on $K$, $\leftharpoonup$ is a right $K$-module coalgebra action on $H$ such that\vspace{-.5em}
\begin{eqnarray}
\label{eq:MP1}
x\rightharpoonup ab&=&(x_1\rightharpoonup a_1)((x_2\leftharpoonup a_2)\rightharpoonup b),\\
\label{eq:MP2}
x\rightharpoonup 1_K&=&\varepsilon_H(x)1_K,\\
\label{eq:MP3}
xy\leftharpoonup a&=&(x\leftharpoonup(y_1\rightharpoonup a_1))(y_2\leftharpoonup a_2),\\
\label{eq:MP4}
1_H\leftharpoonup a&=&\varepsilon_K(a)1_H,\\
\label{eq:MP5}
(x_1\rightharpoonup a_1) \otimes (x_2\leftharpoonup a_2)  &=&
(x_2\rightharpoonup a_2) \otimes (x_1\leftharpoonup a_1)
\end{eqnarray}
for $x,y\in H$ and $a,b\in K$.

For a matched pair of Hopf algebras $(H,K,\rightharpoonup,\leftharpoonup)$, the {\bf double cross product} $K\bowtie H$ is a Hopf algebra structure on $K\otimes H$ equipped with the product
\begin{eqnarray}\label{eq:dcp}
(a\otimes x)(b\otimes y) &\coloneqq& a(x_1\rightharpoonup b_1) \otimes (x_2\leftharpoonup b_2)y,\quad\forall x,y\in H,\ a,b\in K,
\end{eqnarray}
and the usual tensor coproduct. The antipode of $K\bowtie H$ is given by
$$S(a\otimes x)=(1\otimes S(x))(S(a)\otimes 1)=(S(x_1)\rightharpoonup S(a_1))\otimes(S(x_2)\leftharpoonup S(a_2)).$$
\end{defn}

Recently Ferri and Sciandra further considered a certain subclass of matched pairs of Hopf algebras, then introduced two equivalent notions, matched pair of actions on a Hopf algebra and Yetter-Drinfeld brace.

First we recall the notion of matched pair of actions, generalizing the braided group construction of Lu, Yan and Zhu in~\cite{LYZ}.
\begin{defn}[\cite{FS}]
A matched pair $(H,H,\rightharpoonup,\leftharpoonup)$ of Hopf algebras
is called a {\bf matched pair of actions} on $H$, when
\begin{eqnarray}
\label{eq:MP*}
xy&=&(x_1\rightharpoonup y_1)(x_2\leftharpoonup y_2),\quad\forall x,y\in H.
\end{eqnarray}
A matched pair of actions on $H$ will be abbreviated as $(H,\rightharpoonup,\leftharpoonup)$.

\end{defn}

\begin{lemma}
Let $(H,\rightharpoonup,\leftharpoonup)$ be a matched pair of actions on a Hopf algebra $H$. For any $x,y\in H$, we have
\begin{eqnarray}
\label{eq:mp-antipode-1}
S(x\rightharpoonup y) &=& (x\leftharpoonup y_1)\rightharpoonup S(y_2),\\
\label{eq:mp-antipode-2}
S(x\leftharpoonup y) &=& S(x_1) \leftharpoonup(x_2\rightharpoonup y).
\end{eqnarray}
Equivalently,
\begin{eqnarray}
\label{eq:mp-antipode-1'}
S(x_1\leftharpoonup y_1)\rightharpoonup S(x_2\rightharpoonup y_2) &=& \vep(x)S(y),\\
\label{eq:mp-antipode-2'}
S(x_1\leftharpoonup y_1)\leftharpoonup S(x_2\rightharpoonup y_2) &=& \vep(y)S(x).
\end{eqnarray}
\end{lemma}

\begin{proof}
The result of the following convolution product implies Eq.~\eqref{eq:mp-antipode-1}.
 \begin{eqnarray*}
&&((x_1\leftharpoonup y_1)\rightharpoonup S(y_2))(x_2 \rightharpoonup y_3)\\
&=&((x_1\leftharpoonup y_1)\rightharpoonup S(y_4))((x_2\leftharpoonup y_2S(y_3)) \rightharpoonup y_5)\\
&=&((x_1\leftharpoonup y_1)\rightharpoonup S(y_4))(((x_2\leftharpoonup y_2)\leftharpoonup S(y_3)) \rightharpoonup y_5)\\
&\stackrel{\eqref{eq:MP1}}{=}&
(x\leftharpoonup y_1)\rightharpoonup S(y_2)y_3\\
&\stackrel{\eqref{eq:MP2}}{=}&
\vep(x\leftharpoonup y)1\\
&=&\vep(x)\vep(y)1,
\end{eqnarray*}
and then
\begin{eqnarray*}
&&S(x_1\leftharpoonup y_1)\rightharpoonup S(x_2\rightharpoonup y_2) \\
&\stackrel{\eqref{eq:mp-antipode-1}}{=}&
S(x_1\leftharpoonup y_1)\rightharpoonup ((x_2\leftharpoonup y_2)\rightharpoonup S(y_3))\\
&=& S(x_1\leftharpoonup y_1)(x_2\leftharpoonup y_2)\rightharpoonup S(y_3)\\
&=& \vep(x)\vep(y_1)S(y_2)\\
&=& \vep(x)S(y),
\end{eqnarray*}
so Eq.~\eqref{eq:mp-antipode-1'} holds. Eqs.~\eqref{eq:mp-antipode-2} and ~\eqref{eq:mp-antipode-2'} can be proven similarly.
\end{proof}

Next we turn to the notion of Yetter-Drinfeld brace, generalizing cocommutative Hopf brace studied in \cite{AGV}. Recall that a (left-left) {\bf Yetter-Drinfeld module} over a Hopf algebra $H$ is a left $H$-module $M$ which is also a left $H$-comodule satisfying
$$\rho(x\cdot m)=x_1m_{-1}S(x_3)\otimes (x_2\cdot m_0),\quad x\in H,\ m\in M,$$
where $\rho:M\to H\otimes M$ is the coaction map under the notation $\rho(m)=m_{-1}\otimes m_0$.
A morphism of Yetter-Drinfeld modules is a left $H$-module and also left $H$-comodule map. Denote by ${_H^H}\YD$ the braided tensor category of Yetter-Drinfeld modules over $H$.

\begin{defn}[\cite{FS}]\label{def:ydb}
A {\bf Yetter-Drinfeld brace} $(H,\cdot,\circ,1,\Delta,\vep,S,T)$ is the datum of
a Hopf algebra $H_\circ=(H,\circ,1,\Delta,\vep,T)$ and a second multiplication $\cdot$ on $H$ and a linear map $S:H\to H$ such that
\begin{enumerate}[(i)]
\item
$(H,\cdot,1,\Delta,\vep,S)$ is a Hopf algebra in the category ${_{H_\circ}^{H_\circ}}\YD$ via the action $\rightharpoonup$ and the coadjoint coaction $\Ad_L$ respectively defined by
\begin{eqnarray}
\label{eq:ydb-action}
x \rightharpoonup y &\coloneqq& S(x_1)\cdot(x_2 \circ y),\\
\label{eq:ydb-codaction}
\Ad_L(x) &\coloneqq& x_1\circ T(x_3)\otimes x_2
\end{eqnarray}
for any $x,y\in H$.
\item
Define linear map $\leftharpoonup:H\otimes H\to H$ by $x\leftharpoonup y=T(x_1 \rightharpoonup y_1)\circ x_2\circ y_2$, then the pair $(\rightharpoonup,\leftharpoonup)$ satisfies condition \eqref{eq:MP5}.
\item The following Hopf brace compatibility holds.
\begin{eqnarray}\label{eq:hbc}
x\circ(y\cdot z)&=&(x_1\circ y)\cdot S(x_2)\cdot (x_3\circ z),\quad\forall x,y,z\in H.
\end{eqnarray}
\end{enumerate}
\end{defn}

According to \cite[Corollary 3.18,\,Theorem 3.24]{FS}, we know that matched pair of actions on a Hopf algebra and Yetter-Drinfeld brace are two equivalent algebraic structures.
\begin{theorem}[\cite{FS}]
Let $(H,\cdot,1,\Delta,\vep,S)$ be a Hopf algebra with a matched pair of actions $(H,\rightharpoonup,\leftharpoonup)$ on it. Then the tuple
$$(H,\bullet_\rightharpoonup,\cdot,1,\Delta,\vep,S_\rightharpoonup,S)$$
is a Yetter-Drinfeld brace, with linear maps $\bullet_\rightharpoonup:H\otimes H\to H$ and $S_\rightharpoonup:H\to H$ respectively defined by
\begin{eqnarray}
\label{eq:mp-ydb-1}
x\bullet_\rightharpoonup y &=& x_1(S(x_2)\rightharpoonup y),\\
\label{eq:mp-ydb-2}
S_\rightharpoonup(x) &=& x_1\rightharpoonup S(x_2).
\end{eqnarray}
In particular, there is a braided Hopf algebra
$$H_\rightharpoonup=(H,\bullet_\rightharpoonup,1,\Delta,\vep,S_\rightharpoonup)$$
in the category ${_H^H}\YD$ via the action $\rightharpoonup$ and the coadjoint coaction $\Ad_L$, and with its pre-braiding $c_{H_\rightharpoonup,H_\rightharpoonup}$ given by
\begin{eqnarray}\label{eq:ydbraid}
c_{H_\rightharpoonup,H_\rightharpoonup}(x\otimes y) &=& (x_1S(x_3) \rightharpoonup y)\otimes x_2,\quad\forall x,y\in H.
\end{eqnarray}
Namely, all the structure maps of $H_\rightharpoonup$ are morphisms in ${_H^H}\YD$.

Conversely, if $(H,\cdot,\circ,1,\Delta,\vep,S,T)$ is a Yetter-Drinfeld brace, then under the notation in Definition~\ref{def:ydb},
$(H_\circ,\rightharpoonup,\leftharpoonup)$ is a matched pair of actions on $H_\circ$.
\end{theorem}

\begin{remark}
According to \cite[Theorem~6.34]{GGV1}, Yetter-Drinfeld brace is also equivalent to Hopf skew brace introduced by Guccione, Guccione and Valqui. For any Yetter-Drinfeld brace $(H,\cdot,\circ,1,\Delta,\vep,S,T)$, it extends a skew brace with the multiplicative group equal to $(G(H),\circ)$.
\end{remark}

\section{Matched pairs of actions and Yang-Baxter operators}\label{sec:yb}

Guccione, Guccione and Vendramin introduced the notion of braiding operator on cocommutative Hopf algebras in \cite[Definition~5.9]{GGV}, and showed that it provides  Yang-Baxter operators~\cite[Theorem~5.11]{GGV}.
This notion was further applied for arbitrary Hopf algebras in \cite{GGV1}.
\begin{defn}[\cite{GGV1}]\label{def:braiding}
Given a Hopf algebra $H$, a linear operator
$r$ of $H\otimes H$ is called a {\bf braiding operator} on $H$ if it is a coalgebra homomorphism satisfying
\begin{enumerate}[(a)]
\item\label{eq:pb-a}
$mr=m$,
\item\label{eq:pb-b}
$r(m\otimes\id)=(\id\otimes m)(r\otimes \id)(\id\otimes r)$,
\item\label{eq:pb-c}
$r(\id\otimes m)=(m\otimes\id)(\id\otimes r)(r\otimes \id)$,
\item\label{eq:pb-d}
$r(u\otimes \id)=\id\otimes u$,
\item\label{eq:pb-e}
$r(\id\otimes u)=u\otimes \id$,
\end{enumerate}
where $m$ is the multiplication of $H$ and $u$ is the unit map of $H$.
\end{defn}

Referring to \cite[Theorem~5.27,\,Theorem~6.34]{GGV1}, we know that a matched pair of actions $(H,\rightharpoonup,\leftharpoonup)$ on a Hopf algebra $H$ is equivalent to a braiding operator $r$ on $H$ satisfying the braid relation. This equivalence generalizes the original group version in \cite[Theorem~1, Theorem~2]{LYZ} and the cocommutative Hopf algebra case in~\cite[Corollary~2.4]{AGV} and \cite[Theorem~5.11]{GGV}.
Here we reorganize the result as the following theorem and do not claim any originality.
\begin{theorem}\label{thm:yb}
A matched pair of actions $(H,\rightharpoonup,\leftharpoonup)$ on a Hopf algebra $H$ is equivalent to a braiding operator $r$ on $H$, where $r$ is defined by
\begin{eqnarray}\label{eq:pyb}
r(x\otimes y) &\coloneqq& (x_1\rightharpoonup y_1)\otimes (x_2\leftharpoonup y_2),\quad\forall x,y\in H,
\end{eqnarray}
and $(H,\rightharpoonup,\leftharpoonup)$ is conversely recovered from $r$ by letting
\begin{equation}\label{eq:pyb-mpa}
  \rightharpoonup\,\coloneqq(\id\otimes\vep)r,\quad
\leftharpoonup\,\coloneqq(\vep\otimes\id)r.
\end{equation}

Moreover, the braiding operator $r$ on $H$ is a Yang-Baxter operator with its inverse given by
\begin{eqnarray}\label{eq:pyb-left-inv}
r^{-1}(x\otimes y) &=& (y_1 \leftharpoonup(S(y_2)\rightharpoonup S(x_1)))
\otimes ((S(y_3)\leftharpoonup S(x_2))\rightharpoonup x_3),\quad\forall x,y\in H.
\end{eqnarray}
\end{theorem}

\begin{proof}
For the sake of completeness, we reprove the theorem as follows.

First suppose that $(H,\rightharpoonup,\leftharpoonup)$ is a matched pair of actions and $r$ is the linear map defined by Eq.~\eqref{eq:pyb}. Condition~\eqref{eq:MP5} equivalently means that $r$ is a coalgebra homomorphism.
According to~\cite[Lemma~2.4]{FS}, conditions~\eqref{eq:MP1}--\eqref{eq:MP4} hold for $\rightharpoonup$ and $\leftharpoonup$ if and only if the braiding operator conditions \ref{eq:pb-b}--\ref{eq:pb-e} hold for $r$, while condition~\eqref{eq:MP*} is an equivalent form of \ref{eq:pb-a}. So $r$ is a coalgebra homomorphism satisfying all braiding operator conditions \ref{eq:pb-a}--\ref{eq:pb-e}.

Conversely, if $r$ is a braiding operator on $H$, then the linear maps $\rightharpoonup$ and $\leftharpoonup$ defined by \eqref{eq:pyb-mpa} are clearly coalgebra homomorphisms. By~\cite[Lemma~2.7]{FS}, they are also left and right $H$-module actions respectively. Then by the previous discussion, $\rightharpoonup$ and $\leftharpoonup$ are module coalgebra actions satisfying compatibility conditions~\eqref{eq:MP1}--\eqref{eq:MP*}, so $(H,\rightharpoonup,\leftharpoonup)$ is a matched pair of actions.

To check that the braiding operator $r$ satisfies the braid relation, one only need to mimic the proof of \cite[Theorem~5.11]{GGV} based on the equality
$$m^{(2)}(r\otimes \id)(\id\otimes r)(r\otimes \id)=m^{(2)}(\id\otimes r)(r\otimes \id)(\id\otimes r)=m^{(2)}.$$
See also the proof of \cite[Theorem~5.27]{GGV1}.

At last, we check that the braiding operator $r$ has the inverse given by Eq.~\eqref{eq:pyb-left-inv}. Let $t:H\otimes H\to H\otimes H$ be the linear operator defined by \eqref{eq:pyb-left-inv}, and we have
\begin{eqnarray*}
&&t(r(x\otimes y))\\
&\stackrel{\eqref{eq:pyb},\,\eqref{eq:MP5}}{=}&t((x_2\rightharpoonup y_2)\otimes (x_1\leftharpoonup y_1))\\
&\stackrel{\eqref{eq:pyb-left-inv}}{=}&
\big((x_1\leftharpoonup y_1)\leftharpoonup(S(x_2\leftharpoonup y_2)\rightharpoonup S(x_4\rightharpoonup y_4))\big)\\
&&\quad\otimes\,\big((S(x_3\leftharpoonup y_3)\leftharpoonup S(x_5\rightharpoonup y_5))\rightharpoonup (x_6\rightharpoonup y_6)\big)
\\
&\stackrel{\eqref{eq:MP5}}{=}&
\big((x_1\leftharpoonup y_1)\leftharpoonup(S(x_2\leftharpoonup y_2)\rightharpoonup S(x_3\rightharpoonup y_3))\big)\\
&&\quad\otimes\,\big((S(x_4\leftharpoonup y_4)\leftharpoonup S(x_5\rightharpoonup  y_5))\rightharpoonup (x_6\rightharpoonup y_6)\big)\\
&\stackrel{\eqref{eq:mp-antipode-1'},\,\eqref{eq:mp-antipode-2'}}{=}&
\big((x_1\leftharpoonup y_1)\leftharpoonup S(y_2)\big)\otimes \big(S(x_2)\rightharpoonup (x_3\rightharpoonup y_3)\big)\\
&=&\big(x_1\leftharpoonup (y_1S(y_2))\big)\otimes \big((S(x_2)x_3)\rightharpoonup y_3\big)\\
&=& x\otimes y,\\[.5em]
&&r(t(x\otimes y))\\
&\stackrel{\eqref{eq:pyb-left-inv}}{=}& r((y_1 \leftharpoonup(S(y_2)\rightharpoonup S(x_1)))
\otimes ((S(y_3)\leftharpoonup S(x_2))\rightharpoonup x_3))\\
&\stackrel{\eqref{eq:pyb},\,\eqref{eq:MP5}}{=}&
\big((y_2 \leftharpoonup(S(y_3)\rightharpoonup S(x_1)))
\rightharpoonup ((S(y_5)\leftharpoonup S(x_3))\rightharpoonup x_6)\big)\\
&&\quad\otimes\,
\big((y_1 \leftharpoonup(S(y_4)\rightharpoonup S(x_2)))
\leftharpoonup ((S(y_6)\leftharpoonup S(x_4))\rightharpoonup x_5)\big)\\
&=&
\big((y_2 \leftharpoonup(S(y_3)\rightharpoonup S(x_1)))(S(y_5)\leftharpoonup S(x_3))\rightharpoonup x_6\big)\\
&&\quad\otimes\,
\big(y_1 \leftharpoonup (S(y_4)\rightharpoonup S(x_2))((S(y_6)\leftharpoonup S(x_4))\rightharpoonup x_5)\big)\\
&\stackrel{\eqref{eq:MP5}}{=}&
\big((y_2 \leftharpoonup(S(y_3)\rightharpoonup S(x_1)))(S(y_4)\leftharpoonup S(x_2))\rightharpoonup x_6\big)\\
&&\quad\otimes\,
\big(y_1 \leftharpoonup (S(y_5)\rightharpoonup S(x_3))((S(y_6)\leftharpoonup S(x_4))\rightharpoonup x_5)\big)\\
&\stackrel{\eqref{eq:MP5}}{=}&
\big((y_2 \leftharpoonup(S(y_4)\rightharpoonup S(x_2)))(S(y_3)\leftharpoonup S(x_1))\rightharpoonup x_6\big)\\
&&\quad\otimes\,
\big(y_1 \leftharpoonup (S(y_6)\rightharpoonup S(x_4))((S(y_5)\leftharpoonup S(x_3))\rightharpoonup x_5)\big)\\
&\stackrel{\eqref{eq:MP1},\,\eqref{eq:MP3}}{=}&
\big((y_2S(y_3)\leftharpoonup S(x_1))\rightharpoonup x_4\big)\otimes
\big(y_1 \leftharpoonup (S(y_4)\rightharpoonup S(x_2) x_3)\big)\\
&\stackrel{\eqref{eq:MP2},\,\eqref{eq:MP4}}{=}&
\big(\vep(x_1)1\rightharpoonup x_2\big)\otimes
\big(y_1 \leftharpoonup \vep(y_2)1\big)\\
&=&x\otimes y,
\end{eqnarray*}
so $t$ is the inverse of $r$.
\end{proof}

\begin{remark}
In \cite{Ba} Baez introduced the notion of $r$-commutative algebra (or Yang-Baxter commutative algebra), following the idea of Manin~\cite{Man} and with application to quantum group and noncommutative geometry.
Given a Hopf algebra $H$ and a coalgebra endomorphism $r$ of $H\otimes H$, Theorem~\ref{thm:yb} actually tells us that $r$ is a braiding operator on $H$ if and only if $r$ is a Yang-Baxter operator and $H$ is an $r$-commutative algebra with respect to it.
\end{remark}

Next we generalize the formulas in \cite[Proposition 5]{LYZ} to the circumstance of Hopf algebras.
\begin{prop}\label{prop:ybo}
The Yang-Baxter operator $r$ defined by~\eqref{eq:pyb} satisfies
\begin{eqnarray*}
r(S(x_2\leftharpoonup y_2)\otimes S(x_1\rightharpoonup y_1)) &=& S(y)\otimes S(x),\\
r(S(x_1)\otimes (x_2\rightharpoonup y))&=& y_1\otimes S(x\leftharpoonup y_2),\\
r((x \leftharpoonup y_1) \otimes S(y_2))&=& S(x_1 \rightharpoonup y) \otimes x_2
\end{eqnarray*}
for any $x,y\in H$. If the antipode $S$ of $H$ is bijective,
the inverse of $r$ given by \eqref{eq:pyb-left-inv} also has the following form,
\begin{eqnarray}\label{eq:pyb-left-inv'}
r^{-1}(x\otimes y) &=& S^{-1}(S(y_1)\leftharpoonup S(x_1))\otimes S^{-1}(S(y_2)\rightharpoonup S(x_2)).
\end{eqnarray}
\end{prop}
\begin{proof}
First we check that
\begin{eqnarray*}
&&r(S(x_2\leftharpoonup y_2)\otimes S(x_1\rightharpoonup y_1))\\
&\stackrel{\eqref{eq:MP5}}{=}&r(S(x_1\leftharpoonup y_1)\otimes S(x_2\rightharpoonup y_2))\\
&=& (S(x_2\leftharpoonup y_2)\rightharpoonup S(x_4\rightharpoonup y_4))
\otimes (S(x_1\leftharpoonup y_1)\leftharpoonup S(x_3\rightharpoonup y_3))\\
&\stackrel{\eqref{eq:MP5}}{=}&
(S(x_3\leftharpoonup y_3)\rightharpoonup S(x_4\rightharpoonup y_4))
\otimes (S(x_1\leftharpoonup y_1)\leftharpoonup S(x_2\rightharpoonup y_2))\\
&\stackrel{\eqref{eq:mp-antipode-1'},\,\eqref{eq:mp-antipode-2'}}{=}&
\vep(x_2)S(y_2)\otimes \vep(y_1) S(x_1)\\
&=& S(y)\otimes S(x).
\end{eqnarray*}
Namely, $r(S\otimes S)\tau r=(S\otimes S)\tau$, where $\tau$ is the flip map. Also,
\begin{eqnarray*}
r(S(x_1)\otimes (x_2\rightharpoonup y))&=& (S(x_2)\rightharpoonup (x_3\rightharpoonup y_1))\otimes (S(x_1)\leftharpoonup (x_4\rightharpoonup y_2))\\
&=& (S(x_2)x_3\rightharpoonup y_1)\otimes (S(x_1)\leftharpoonup (x_4\rightharpoonup y_2))\\
&=& y_1\otimes (S(x_1)\leftharpoonup (x_2\rightharpoonup y_2))\\
&\stackrel{\eqref{eq:mp-antipode-2}}{=}& y_1\otimes S(x\leftharpoonup y_2),\\[.5em]
r((x \leftharpoonup y_1) \otimes S(y_2))&=& ((x_1 \leftharpoonup y_1)\rightharpoonup S(y_4)) \otimes ((x_2 \leftharpoonup y_2) \leftharpoonup S(y_3))\\
&=& ((x_1 \leftharpoonup y_1)\rightharpoonup S(y_4)) \otimes (x_2 \leftharpoonup y_2S(y_3))\\
&=& ((x_1 \leftharpoonup y_1)\rightharpoonup S(y_2)) \otimes x_2\\
&\stackrel{\eqref{eq:mp-antipode-1}}{=}& S(x_1 \rightharpoonup y) \otimes x_2.
\end{eqnarray*}
In particular, if $S$ is bijective, the inverse
$r^{-1}=(S^{-1}\otimes S^{-1})\tau r (S\otimes S)\tau$ is as stated in \eqref{eq:pyb-left-inv'}.
\end{proof}

Now we are in the position to give a positive answer for the problem raised by Ferri and Sciandra~\cite[Problem~5.9]{FS}. It generalizes the case of braided groups in \cite[Proposition~ 4]{LYZ} and that of cocommutative Hopf braces in \cite[Corollary~2.5]{AGV}.
\begin{theorem}\label{thm:involutive}
Let $(H,\rightharpoonup,\leftharpoonup)$ be a matched pair of actions on a Hopf algebra $H$. The following conditions are equivalent:
\begin{enumerate}[(i)]
\item\label{in1}
The Yang-Baxter operator $r$ defined by~\eqref{eq:pyb} is involutive.
\item\label{in2}
For any $x,y\in H$,
\begin{eqnarray}
\label{eq:in2}
(x_1\rightharpoonup y_1)\rightharpoonup (x_2\leftharpoonup y_2)=\vep(y)x,\quad
(x_1\rightharpoonup y_1)\leftharpoonup (x_2\leftharpoonup y_2)=\vep(x)y.
\end{eqnarray}
\item\label{in3}
For any $x,y\in H$,
\begin{eqnarray}
\label{eq:in3}
x\leftharpoonup y &=& S(x_1\rightharpoonup y)\rightharpoonup x_2.
\end{eqnarray}
\item\label{in4}
The multiplication
$m_{\bullet_\rightharpoonup}$ of $H_\rightharpoonup$ defined by \eqref{eq:mp-ydb-1} is braided commutative in ${_H^H}\YD$, namely,
$$m_{\bullet_\rightharpoonup}  c_{H_\rightharpoonup,H_\rightharpoonup}=m_{\bullet_\rightharpoonup}.$$
\end{enumerate}

\end{theorem}

\begin{proof}
First we show that \ref{in1} and \ref{in2} are equivalent. Note that
\begin{eqnarray*}
r^2(x\otimes y)
&=&((x_1\rightharpoonup y_1)\rightharpoonup (x_3\leftharpoonup y_3))\otimes
((x_2\rightharpoonup y_2)\leftharpoonup (x_4\leftharpoonup y_4))\\
&\stackrel{\eqref{eq:MP5}}{=}& ((x_1\rightharpoonup y_1)\rightharpoonup (x_2\leftharpoonup y_2))\otimes
((x_3\rightharpoonup y_3)\leftharpoonup (x_4\leftharpoonup y_4)),
\end{eqnarray*}
and we have
\begin{eqnarray*}
(\id\otimes\vep)r^2(x\otimes y)&=&(x_1\rightharpoonup y_1)\rightharpoonup (x_2\leftharpoonup y_2),\\
(\vep\otimes\id)r^2(x\otimes y)&=&(x_1\rightharpoonup y_1)\leftharpoonup (x_2\leftharpoonup y_2).
\end{eqnarray*}
So $r^2=\id^{\otimes2}$ equivalently implies that Eq.~\eqref{eq:in2}.

Next we show that \ref{in2} and \ref{in3} are equivalent. If \ref{in2} is true, then
\begin{eqnarray*}
x\leftharpoonup y &=& S(x_1\rightharpoonup y_1)(x_2\rightharpoonup y_2) \rightharpoonup (x_3\leftharpoonup y_3)\\
&=&
S(x_1\rightharpoonup y_1)\rightharpoonup((x_2\rightharpoonup y_2) \rightharpoonup (x_3\leftharpoonup y_3))\\
&\stackrel{\eqref{eq:in2}}{=}&
S(x_1\rightharpoonup y_1)\rightharpoonup \vep(y_2)x_2\\
&=&
S(x_1\rightharpoonup y)\rightharpoonup x_2.
\end{eqnarray*}
Conversely, if \ref{in3} is true, then
\begin{eqnarray*}
(x_1\rightharpoonup y_1)\rightharpoonup (x_2\leftharpoonup y_2)
&\stackrel{\eqref{eq:in3}}{=}& (x_1\rightharpoonup y_1)\rightharpoonup (S(x_2\rightharpoonup y_2)\rightharpoonup x_3)\\
&=& (x_1\rightharpoonup y_1)S(x_2\rightharpoonup y_2)\rightharpoonup x\\
&=& \vep(y)x,\\[.3em]
(x_1\rightharpoonup y_1)\leftharpoonup (x_2\leftharpoonup y_2)
&\stackrel{\eqref{eq:in3}}{=}& S((x_1\rightharpoonup y_1)\rightharpoonup (x_3\leftharpoonup y_3))\rightharpoonup (x_2\rightharpoonup y_2)\\
&\stackrel{\eqref{eq:MP5}}{=}& S((x_1\rightharpoonup y_1)\rightharpoonup (x_2\leftharpoonup y_2))\rightharpoonup (x_3\rightharpoonup y_3)\\
&=& S(\vep(y_1)x_1)\rightharpoonup (x_2\rightharpoonup y_2)\\
&=& S(x_1)x_2\rightharpoonup y\\
&=& \vep(x)y.
\end{eqnarray*}

At last, we show that \ref{in3} and \ref{in4} are equivalent. If \ref{in3} is true, then
\begin{eqnarray*}
m_{\bullet_\rightharpoonup}(c_{H_\rightharpoonup,H_\rightharpoonup}(x\otimes y))
&\stackrel{\eqref{eq:ydbraid}}{=}&
(x_1S(x_3)\rightharpoonup y)\bullet_\rightharpoonup x_2,\\
&\stackrel{\eqref{eq:mp-ydb-1}}{=}& (x_1S(x_5)\rightharpoonup y_1)(S(x_2S(x_4)\rightharpoonup y_2)\rightharpoonup x_3)\\
&=& (x_1\rightharpoonup (S(x_5)\rightharpoonup y_1))(S(x_2\rightharpoonup (S(x_4)\rightharpoonup y_2))\rightharpoonup x_3)\\
&\stackrel{\eqref{eq:in3}}{=}& (x_1\rightharpoonup (S(x_4)\rightharpoonup y_1))(x_2\leftharpoonup (S(x_3)\rightharpoonup y_2))\\
&\stackrel{\eqref{eq:MP*}}{=}& x_1(S(x_2)\rightharpoonup y)\\
&\stackrel{\eqref{eq:mp-ydb-1}}{=}& x\bullet_\rightharpoonup y,
\end{eqnarray*}
so $m_{\bullet_\rightharpoonup} c_{H_\rightharpoonup,H_\rightharpoonup}=m_{\bullet_\rightharpoonup}$.
Conversely, we have
\begin{eqnarray*}
m_{\bullet_\rightharpoonup}(x_1\otimes(x_2\rightharpoonup y)) &\stackrel{\eqref{eq:mp-ydb-1}}{=}& x_1(S(x_2)\rightharpoonup(x_3\rightharpoonup y))\\
&=& x_1(S(x_2)x_3\rightharpoonup y)\\
&=& xy\\
&\stackrel{\eqref{eq:MP*}}{=}&
(x_1\rightharpoonup y_1)(x_2\leftharpoonup y_2),\\[.3em]
m_{\bullet_\rightharpoonup}(c_{H_\rightharpoonup,H_\rightharpoonup}(x_1\otimes(x_2\rightharpoonup y))
&\stackrel{\eqref{eq:ydbraid}}{=}& (x_1S(x_3)\rightharpoonup (x_4\rightharpoonup y))\bullet_\rightharpoonup x_2,\\
&=& (x_1S(x_3)x_4\rightharpoonup y)\bullet_\rightharpoonup x_2,\\
&=& (x_1\rightharpoonup y)\bullet_\rightharpoonup x_2\\
&\stackrel{\eqref{eq:mp-ydb-1}}{=}& (x_1\rightharpoonup y_1)(S(x_2\rightharpoonup y_2) \rightharpoonup x_3).
\end{eqnarray*}
When $m_{\bullet_\rightharpoonup}  c_{H_\rightharpoonup,H_\rightharpoonup}=m_{\bullet_\rightharpoonup}$, we cancel factors via convolution product to obtain Eq.~\eqref{eq:in3}.
\end{proof}

\begin{remark}
When $H$ is non-cocommutative, one can not use the right adjoint action of $H_\rightharpoonup$ to define a Yang-Baxter operator equivalent to $r$ as in the cocommutative case~(\cite[Theorem~2.3]{AGV}). It urges us to solve the problem in~\cite[Problem~5.9]{FS} directly by Theorem~\ref{thm:involutive}.
\end{remark}

Applying Theorem~\ref{thm:involutive}, we generalize the formulas given in \cite[Proposition~2.4]{Ta} as follows.
\begin{coro}
Let $(H,\rightharpoonup,\leftharpoonup)$ be a matched pair of actions on a Hopf algebra $H$. If the Yang-Baxter operator $r$ associated to $(H,\rightharpoonup,\leftharpoonup)$ is involutive, then for any $x,y\in H$,
\begin{eqnarray*}
S(x\leftharpoonup y)&=&S(y)\rightharpoonup S(x),\\
S(x\rightharpoonup y)&=&S(y)\leftharpoonup S(x).
\end{eqnarray*}
\end{coro}
\begin{proof}
When $r$ is involutive, we  actually have
\begin{eqnarray*}
S(x\leftharpoonup y)&\stackrel{\eqref{eq:mp-antipode-2}}{=}& S(x_1) \leftharpoonup(x_2\rightharpoonup y)\\
&\stackrel{\eqref{eq:in3}}{=}& S(S(x_2) \rightharpoonup(x_3\rightharpoonup y))\rightharpoonup S(x_1)\\
&=& S(S(x_2)x_3\rightharpoonup y)\rightharpoonup S(x_1)\\
&=& S(y)\rightharpoonup S(x).
\end{eqnarray*}
By symmetry, we also have $S(x\rightharpoonup y)=S(y)\leftharpoonup S(x)$.
\end{proof}

Next we compute
the two adjoint actions of the braided Hopf algebra $H_\rightharpoonup$ in the category ${_H^H}\YD$ to justify Theorem~\ref{thm:involutive}.
\begin{prop}\label{prop:adjoint}
Let $(H,\rightharpoonup,\leftharpoonup)$ be a matched pair of actions on a Hopf algebra $H$. The following formulas hold for the left adjoint action and the right adjoint action of $H_\rightharpoonup$.
\begin{eqnarray*}
\ad_{L,x}(y) &=& x_1(S(x_4)\rightharpoonup y_1)(S(S(x_3)\rightharpoonup y_2)\rightharpoonup S(x_2)),\\
\ad_{R,y}(x) &=& (x_1\rightharpoonup (S(x_4) \rightharpoonup y_1))\rightharpoonup(x_2\leftharpoonup (S(x_3)\rightharpoonup y_2))
\end{eqnarray*}
for any $x,y\in H$.
\end{prop}
\begin{proof}
By definition, we compute the left adjoint action of $H_\rightharpoonup$ as follows.
\begin{eqnarray*}
\ad_{L,x}(y) &=& m_{\bullet_\rightharpoonup}^{(2)}(\id^{\otimes 2}\otimes S_\rightharpoonup)(\id\otimes c_{H,H})(\Delta\otimes\id)(x\otimes y)\\
&\stackrel{\eqref{eq:ydbraid}}{=}& x_1\bullet_\rightharpoonup(x_2S(x_4)\rightharpoonup y)\bullet_\rightharpoonup S_\rightharpoonup(x_3)\\
&\stackrel{\eqref{eq:mp-ydb-1}}{=}& (x_1(S(x_2)\rightharpoonup (x_3S(x_5)\rightharpoonup y)))\bullet_\rightharpoonup S_\rightharpoonup(x_4)\\
&=& (x_1(S(x_2)x_3S(x_5)\rightharpoonup y))\bullet_\rightharpoonup S_\rightharpoonup(x_4)\\
&=& (x_1(S(x_3)\rightharpoonup y))\bullet_\rightharpoonup S_\rightharpoonup(x_2)\\
&\stackrel{\eqref{eq:mp-ydb-1}}{=}& x_1(S(x_5)\rightharpoonup y_1)(S(x_2(S(x_4)\rightharpoonup y_2))\rightharpoonup S_\rightharpoonup(x_3))\\
&\stackrel{\eqref{eq:mp-ydb-2}}{=}& x_1(S(x_6)\rightharpoonup y_1)(S(S(x_5)\rightharpoonup y_2)S(x_2)\rightharpoonup (x_3\rightharpoonup S(x_4)))\\
&=& x_1(S(x_6)\rightharpoonup y_1)(S(S(x_5)\rightharpoonup y_2)S(x_2)x_3\rightharpoonup S(x_4))\\
&=& x_1(S(x_4)\rightharpoonup y_1)(S(S(x_3)\rightharpoonup y_2)\rightharpoonup S(x_2)).
\end{eqnarray*}
Next we compute the right adjoint action of $H_\rightharpoonup$ as follows.
\begin{eqnarray*}
\ad_{R,y}(x) &=& m_{\bullet_\rightharpoonup}^{(2)}(S_\rightharpoonup\otimes \id^{\otimes 2})(c_{H,H}\otimes \id)(\id\otimes\Delta)(x\otimes y)\\
&\stackrel{\eqref{eq:ydbraid}}{=}& S_\rightharpoonup(x_1S(x_3)\rightharpoonup y_1)\bullet_\rightharpoonup x_2 \bullet_\rightharpoonup y_2\\
&=& (x_1S(x_3)\rightharpoonup S_\rightharpoonup(y_1))\bullet_\rightharpoonup x_2 \bullet_\rightharpoonup y_2\\
&\stackrel{\eqref{eq:mp-ydb-1}}{=}&
(x_1S(x_4)\rightharpoonup S_\rightharpoonup(y_1))\bullet_\rightharpoonup (x_2(S(x_3) \rightharpoonup y_2))\\
&\stackrel{\eqref{eq:MP*}}{=}&
(x_1S(x_6)\rightharpoonup S_\rightharpoonup(y_1))\bullet_\rightharpoonup ((x_2S(x_5) \rightharpoonup y_2)(x_3\leftharpoonup (S(x_4)\rightharpoonup y_3)))\\
&\stackrel{\eqref{eq:mp-ydb-1}}{=}&
(x_1S(x_8)\rightharpoonup S_\rightharpoonup(y_1))\bullet_\rightharpoonup
(x_2S(x_7) \rightharpoonup y_2) \\
&&\bullet_\rightharpoonup
((x_3S(x_6) \rightharpoonup y_3)\rightharpoonup(x_4\leftharpoonup (S(x_5)\rightharpoonup y_4)))\\
&=& (x_1S(x_6)\rightharpoonup S_\rightharpoonup(y_1)\bullet_\rightharpoonup y_2)\bullet_\rightharpoonup
((x_2S(x_5) \rightharpoonup y_3)\rightharpoonup(x_3\leftharpoonup (S(x_4)\rightharpoonup y_4)))\\
&\stackrel{\eqref{eq:MP2}}{=}&\vep(x_1S(x_6))((x_2S(x_5) \rightharpoonup y_1)\rightharpoonup(x_3\leftharpoonup (S(x_4)\rightharpoonup y_2)))\\
&=& (x_1\rightharpoonup (S(x_4) \rightharpoonup y_1))\rightharpoonup(x_2\leftharpoonup (S(x_3)\rightharpoonup y_2)),
\end{eqnarray*}
where the third and the seventh equalities are due to the fact that
$S_\rightharpoonup$ and $\bullet_\rightharpoonup$ are morphisms in the category ${_H^H}\YD$.
\end{proof}

\begin{remark}
When $H_\rightharpoonup$ is braided commutative, we can apply Theorem~\ref{thm:involutive} to show that the two adjoint actions of $H_\rightharpoonup$ are trivial. Indeed,
\begin{eqnarray*}
\ad_{L,x}(y) &=& x_1(S(x_4)\rightharpoonup y_1)(S(S(x_3)\rightharpoonup y_2)\rightharpoonup S(x_2))\\
&\stackrel{\eqref{eq:in3}}{=}& x_1(S(x_3)\rightharpoonup y_1)(S(x_2)\leftharpoonup y_2)\\
&\stackrel{\eqref{eq:MP*}}{=}& x_1S(x_2)y\ =\ \vep(x)y,\\
\ad_{R,y}(x) &=& (x_1\rightharpoonup (S(x_4) \rightharpoonup y_1))\rightharpoonup(x_2\leftharpoonup (S(x_3)\rightharpoonup y_2))\\
&\stackrel{\eqref{eq:in2}}{=}& x_1\vep(S(x_2)\rightharpoonup y)\ =\ x\vep(y).
\end{eqnarray*}
\end{remark}

\begin{defn}[\cite{Ra,Ma0}]
Given a braided Hopf algebra $R$ in the category ${_H^H}\YD$ of Yetter-Drinfeld modules over a Hopf algebra $H$, there is an ordinary Hopf algebra structure on the tensor product $R\otimes H$, called the {\bf biproduct} or {\bf bosonization} of $R$ by $H$.
It is the smash product Hopf algebra $R\# H$ equipped with the multiplication
\begin{eqnarray}\label{eq:bp_mul}
(a\otimes x)(b\otimes y)&\coloneqq&a(x_1\cdot b)\otimes x_2y,\quad \forall a,b\in R,\,x,y\in H,
\end{eqnarray}
and the following coproduct
\begin{eqnarray}\label{eq:bp_co}
\Delta(a\otimes x)&\coloneqq&(a^1\otimes (a^2)_{-1}x_1)\otimes ((a^2)_0\otimes x_2),
\end{eqnarray}
where the coproduct of $R$ is written as $\Delta_R(a)=a^1\otimes a^2$.
\end{defn}

In the end of this section, we identify the double cross product $H\bowtie H$ with the bosonization of $H_\rightharpoonup$ by $H$ for a matched pair of actions $(H,\rightharpoonup,\leftharpoonup)$. Namely, $H\bowtie H$ becomes a Hopf algebra with a projection and $H_\rightharpoonup$ serves as its subalgebra of coinvariants.
\begin{theorem}\label{thm:dcp_bos}
Let $(H,\rightharpoonup,\leftharpoonup)$ be a matched pair of actions on a Hopf algebra $H$. The following linear map
$$\Phi:H\bowtie H\to H_\rightharpoonup\# H,\ x\otimes y\mapsto x_1\otimes x_2y$$
gives an isomorphism of Hopf algebras between the double cross poduct $H\bowtie H$ and the bosonization  $H_\rightharpoonup\# H$.
\end{theorem}
\begin{proof}
It is clear that $\Phi$ has the inverse $\Phi^{-1}:x\otimes y\mapsto x_1\otimes S(x_2)y$.
For any $x,y,z,w\in H$,
\begin{eqnarray*}
\Phi((x\otimes y)(z\otimes w))&\stackrel{\eqref{eq:dcp}}{=}& \Phi(x(y_1\rightharpoonup z_1)\otimes(y_2 \leftharpoonup z_2)w)\\
&=& x_1(y_1\rightharpoonup z_1)\otimes x_2(y_2\rightharpoonup z_2)(y_3 \leftharpoonup z_3)w\\
&\stackrel{\eqref{eq:MP*}}{=}& x_1(y_1\rightharpoonup z_1)\otimes x_2y_2z_2w.
\end{eqnarray*}
On the other hand,
\begin{eqnarray*}
\Phi(x\otimes y)\Phi(z\otimes w)&=& (x_1\otimes x_2y)(z_1\otimes z_2w)\\
&\stackrel{\eqref{eq:bp_mul}}{=}&
x_1\bullet_\rightharpoonup(x_2y_1\rightharpoonup z_1)\otimes x_3y_2z_2w\\
&\stackrel{\eqref{eq:mp-ydb-1}}{=}&
x_1(S(x_2)\rightharpoonup (x_3y_1\rightharpoonup z_1))\otimes x_4y_2z_2w\\
&=& x_1(S(x_2)x_3y_1\rightharpoonup z_1)\otimes x_4y_2z_2w\\
&=& x_1(y_1\rightharpoonup z_1)\otimes x_2y_2z_2w,
\end{eqnarray*}
so $\Phi$ is an algebra homomorphism. Meanwhile,
\begin{eqnarray*}
\Delta(\Phi(x\otimes y))&=&\Delta(x_1\otimes x_2y)\\
&\stackrel{\eqref{eq:bp_co},\,\eqref{eq:ydb-codaction}}{=}&
(x_1\otimes (x_2S(x_4))(x_5y_1))\otimes (x_3\otimes x_6y_2)\\
&=& (x_1\otimes x_2y_1)\otimes (x_3\otimes x_4y_2).
\end{eqnarray*}
On the other hand,
\begin{eqnarray*}
(\Phi\otimes \Phi)(\Delta(x\otimes y))&=&
\Phi(x_1\otimes y_1)\otimes \Phi(x_2\otimes y_2)\\
&=& (x_1\otimes x_2y_1)\otimes (x_3\otimes x_4y_2).
\end{eqnarray*}
Namely, $\Phi$ is also a coalgebra homomorphism, thus an isomorphism of Hopf algebras.
\end{proof}

\section{An illustration to classify matched pairs of actions}\label{sec:construct}

According to \cite[Theorem~5.27,\,Theorem~6.34]{GGV1}, we know that a matched pair of actions on a Hopf algebra $H$ is just a pair of module coalgebra actions $(\rightharpoonup,\leftharpoonup)$ satisfying compatibility condition \eqref{eq:MP*}, namely conditions \eqref{eq:MP1}--\eqref{eq:MP5} are redundant. In other words, we have the following simplified characterization of matched pairs of actions on Hopf algebras for the purpose of classification.
\begin{theorem}\label{thm:construct-mp}
Let $H$ be a Hopf algebra, and $\rightharpoonup$ be a left $H$-module coalgebra action on itself. Define linear map $\leftharpoonup:H\otimes H\to H$ by
\begin{eqnarray}\label{eq:r-action}
x\leftharpoonup y&=&S(x_1\rightharpoonup y_1)x_2y_2,\quad\forall x,y\in H.
\end{eqnarray}
If $\leftharpoonup$ is a right $H$-module coalgebra action,
then $(H,\rightharpoonup,\leftharpoonup)$
is a matched pair of actions on $H$.
\end{theorem}
\begin{proof}
For the sake of completeness, we give the following sketched proof. Suppose that $\leftharpoonup$ defined by \eqref{eq:r-action} is a right $H$-module coalgebra action.
The defining formula~\eqref{eq:r-action} equivalently implies condition \eqref{eq:MP*}. Also, $\vep(x\leftharpoonup y)=\vep(x)\vep(y)$ for any $x,y\in H$, since $\rightharpoonup$ is a left $H$-module coalgebra action. According to \cite[Lemma~3.6]{FS}, $\leftharpoonup$ is a coalgebra homomorphism if and only if condition \eqref{eq:MP5} holds.

Next we show that $\rightharpoonup$ is a left ``twisted'' module algebra action satisfying conditions \eqref{eq:MP1} and \eqref{eq:MP2}.
For any $x,y,z\in H$,
\begin{eqnarray*}
xyz &=& x(yz)\\
&\stackrel{\eqref{eq:MP*}}{=}&
(x_1\rightharpoonup y_1z_1)(x_2\leftharpoonup y_2z_2)\\
&=&(x_1\rightharpoonup y_1z_1)((x_2\leftharpoonup y_2)\leftharpoonup z_2).
\end{eqnarray*}
On the other hand,
\begin{eqnarray*}
xyz &=& (xy)z\\
&\stackrel{\eqref{eq:MP*}}{=}&
(x_1\rightharpoonup y_1)(x_2\leftharpoonup y_2)z\\
&\stackrel{\eqref{eq:MP*}}{=}& (x_1\rightharpoonup y_1)
((x_2\leftharpoonup y_2)\rightharpoonup  z_1)
((x_3\leftharpoonup y_3)\leftharpoonup z_2),
\end{eqnarray*}
so we see that condition \eqref{eq:MP1} holds by cancelling factors via convolution product.

Meanwhile, we have
\begin{eqnarray*}
\vep(x_1)1x_2&=&  x1\\
&\stackrel{\eqref{eq:MP*}}{=}&(x_1\rightharpoonup 1)(x_2\leftharpoonup 1)\\
&=&(x_1\rightharpoonup 1)x_2.
\end{eqnarray*}
Hence, condition \eqref{eq:MP2} also holds by cancelling factors via convolution product.

It  remains to check that $\leftharpoonup$ is a right ``twisted'' module algebra action satisfying conditions \eqref{eq:MP3}, \eqref{eq:MP4}. Indeed, for any $x,y,z\in H$,
\begin{eqnarray*}
1\leftharpoonup x&\stackrel{\eqref{eq:r-action}}{=}& S(1\rightharpoonup x_1)x_2\ =\ S(x_1)x_2\ =\ \vep(x)1\\[.5em]
(xy)\leftharpoonup z &\stackrel{\eqref{eq:r-action}}{=}& S(x_1y_1\rightharpoonup z_1)x_2y_2z_2\\
&\stackrel{\eqref{eq:MP*}}{=}& S(x_1\rightharpoonup(y_1\rightharpoonup z_1))x_2(y_2\rightharpoonup z_2)(y_3\leftharpoonup z_3)\\
&\stackrel{\eqref{eq:r-action}}{=}&
(x\leftharpoonup(y_1\rightharpoonup z_1))(y_2\leftharpoonup z_2).
\end{eqnarray*}
In summary, $(H,\rightharpoonup,\leftharpoonup)$
is a matched pair of actions on $H$.
\end{proof}

Relating to Majid's transmutation theory, the authors in \cite{FS} pointed out that coquasitriangular structures on Hopf algebras induce matched pairs of actions, i.e. Yetter-Drinfeld braces, and applied this result to several examples, such as
Sweedler's Hopf algebra $H_4$ and its generalization $E_n\ (n\geq 3)$, and also quantum group ${\rm SL}_q(2)$.

Recall that a {\bf coquasitriangular Hopf algebra} is a Hopf algebra $H$ with a convolution-invertible bilinear map $\calr:H\otimes H\to\bk$ such that
\begin{eqnarray*}
\calr(a_1\otimes b_1)a_2b_2&=&b_1a_1\calr(a_2\otimes b_2),\\
\calr(a\otimes bc)&=&\calr(a_1\otimes c)\calr(a_2\otimes b),\\
\calr(ab\otimes c)&=&\calr(a\otimes c_1)\calr(b\otimes c_2)
\end{eqnarray*}
for any $a,b,c\in H$. In particular, if $\calr^{-1}=\calr^{\rm op}$, it is called {\bf cotriangular}~\cite[\S~2.2]{Ma}.
\begin{theorem}[\cite{FS}]\label{thm:coq-mpa}
Any coquasitriangular structure $\calr$ on a Hopf algebra $H$ naturally induces a matched pair of actions $(H,\rightharpoonup,\leftharpoonup)$ defined by
\begin{eqnarray*}
a\rightharpoonup b&=& \calr^{-1}(a_1\otimes b_1)b_2\calr(a_2\otimes b_3),\\
a\leftharpoonup b&=& \calr^{-1}(a_1\otimes b_1)a_2\calr(a_3\otimes b_2).
\end{eqnarray*}
The associated braiding operator $r$ as in Eq.~\eqref{eq:pyb} is given by
\begin{eqnarray}\label{eq:cqt-braid}
r(a\otimes b)
&=&\calr^{-1}(a_1\otimes b_1)b_2\otimes a_2\calr(a_3\otimes b_3).
\end{eqnarray}
\end{theorem}

Here as an illustration, we use the characterization in Theorem~\ref{thm:construct-mp} to classify matched pairs of actions on the $8$-dimensional non-semisimple Hopf algebra $A_{C_2\times C_2}$ intensively studied in \cite{St}. After that, we analyze their associated Yang-Baxter operators.

\begin{exam}\label{ex:A22}
The Hopf algebra $A_{C_2\times C_2}$ is generated by 3 elements $g$, $h$, $x$, with the relations
$$g^{2} =h^{2}= 1, \quad gh = hg,\quad x^{2} =0,\quad gx=-xg,\quad hx=-xh,$$
A linear basis for $A_{C_2\times C_2}$ is given by $\{g^ih^jx^k\,|\,i,j,k=0,1\}$. The coalgebra structure and the antipode of $A_{C_2\times C_2}$ are defined by
$$
\Delta(g) = g \otimes g, \quad \Delta(h) = h \otimes h,
\quad \Delta(x) =x\otimes 1+g\otimes x,$$
$$
\varepsilon(g) = \varepsilon(h) = 1, \quad\varepsilon(x) = 0, \quad S(g) = g, \quad S(h) = h, \quad S(x) = xg,
$$
so it contains $H_4$ as a Hopf subalgebra. In particular, the group of group-like elements in $A_{C_2\times C_2}$
$$G(A_{C_2\times C_2})=\{1,g,h,gh\},$$
isomorphic to the Klein 4-group, and all the nonzero subspaces of skew primitive elements in $A_{C_2\times C_2}$ are as follows.
\begin{align*}
&P_{1,h}(A_{C_2\times C_2})=P_{h,1}(A_{C_2\times C_2})=\bk(1-h),\\
&P_{1,gh}(A_{C_2\times C_2})=P_{gh,1}(A_{C_2\times C_2})=\bk(1-gh),\\
&P_{g,h}(A_{C_2\times C_2})=P_{h,g}(A_{C_2\times C_2})=\bk(g-h),\\
&P_{g,gh}(A_{C_2\times C_2})=P_{gh,g}(A_{C_2\times C_2})=\bk(g-gh),\\
&P_{1,g}(A_{C_2\times C_2})=\bk(1-g)\oplus \bk x,\\
&P_{g,1}(A_{C_2\times C_2})=\bk(1-g)\oplus \bk gx,\\
&P_{h,gh}(A_{C_2\times C_2})=\bk(h-gh)\oplus \bk hx,\\
&P_{gh,h}(A_{C_2\times C_2})=\bk(h-gh)\oplus \bk ghx.
\end{align*}

Suppose that there is a left module coalgebra action $\rightharpoonup$ on $A_{C_2\times C_2}$,
and the linear map $\leftharpoonup$ defined by Eq.~\eqref{eq:r-action} is a right module coalgebra action on $A_{C_2\times C_2}$. Then
\begin{align*}
&g\leftharpoonup g = g\rightharpoonup g,\\
&g\leftharpoonup h = (g\rightharpoonup h)gh,\\
&h\leftharpoonup g = (h\rightharpoonup g)gh,\\
&h\leftharpoonup h = h\rightharpoonup h
\end{align*}
are all non-trivial group-like elements by Eqs.~\eqref{eq:MP2}, \eqref{eq:MP4}. Otherwise,
let $g\leftharpoonup g=1$ for instance, then $g=g\leftharpoonup g^2=(g\leftharpoonup g)\leftharpoonup g=1\leftharpoonup g=1$, which is a contradiction.

First note that
$g\rightharpoonup x\in P_{1,g\rightharpoonup g}(A_{C_2\times C_2})\setminus\bk G(A_{C_2\times C_2})$, as $g\rightharpoonup (g\rightharpoonup x)=x$.
According to the situation of skew primitive elements of $A_{C_2\times C_2}$, we have $g\rightharpoonup g=g\leftharpoonup g = g$. Then it implies that $x\rightharpoonup g$, $g\leftharpoonup x\in P_{g,g}(A_{C_2\times C_2})=\{0\}$, so $x\rightharpoonup g=g\leftharpoonup x=0$.
Meanwhile, for $g\rightharpoonup x\in P_{1,g}(A_{C_2\times C_2})=\bk(1-g)\oplus \bk x$, we know that it is a scalar multiple of $x$, and then
$$g\leftharpoonup x \stackrel{\eqref{eq:r-action}}{=} S(g\rightharpoonup x)g+S(g\rightharpoonup g)gx
=(g\rightharpoonup x)g^2+g^2x
=g\rightharpoonup x+x=0.$$
So $g\rightharpoonup x=-x$. On the other hand, $x\leftharpoonup g \stackrel{\eqref{eq:r-action}}{=} S(x\rightharpoonup g)g+S(g\rightharpoonup g)xg=gxg=-x$.

By a similar argument, we consider
$h\rightharpoonup x\in P_{1,h\rightharpoonup g}(A_{C_2\times C_2})$ to obtain that $h\rightharpoonup g=g$, and then $h\leftharpoonup g=h$. It implies that  $h\leftharpoonup x\in P_{h,h}(A_{C_2\times C_2})=0$, so $h\leftharpoonup x=0$.
For $h\rightharpoonup x\in P_{1,g}(A_{C_2\times C_2})=\bk(1-g)\oplus \bk x$, it is also a scalar multiple of $x$, and then
$$h\leftharpoonup x \stackrel{\eqref{eq:r-action}}{=} S(h\rightharpoonup x)h+S(h\rightharpoonup g)hx
=(h\rightharpoonup x)gh+ghx=0,$$
so $h\rightharpoonup x=-x$. By symmetry, we have $g\rightharpoonup h=h$, $g\leftharpoonup h=g$, $x\rightharpoonup h=0$ and $x\leftharpoonup h=-x$.

Now since
\begin{align*}
\Delta(x\rightharpoonup x)
&=(x\rightharpoonup x)\otimes (1\rightharpoonup 1)+ (g\rightharpoonup x)\otimes (x\rightharpoonup 1)\\
&\quad +(x\rightharpoonup g)\otimes (1\rightharpoonup x) + (g\rightharpoonup g)\otimes (x\rightharpoonup x)\\
&=(x\rightharpoonup x)\otimes 1+ g\otimes (x\rightharpoonup x),
\end{align*}
we know that $x\rightharpoonup x\in P_{1,g}(A_{C_2\times C_2})=\bk(1-g)\oplus \bk x$.
For $x\rightharpoonup(x\rightharpoonup x)=x^2\rightharpoonup x=0$, we should have $x\rightharpoonup x=\alpha(1-g)$ for some $\alpha\in \bk$.  Then
$$x\leftharpoonup x \stackrel{\eqref{eq:r-action}}{=} S(x\rightharpoonup x)+S(g\rightharpoonup x)x + S(x\rightharpoonup g)x+S(g\rightharpoonup g)x^2
=\alpha(1-g).$$

Consequently, we have confirmed that
$$\left\{\begin{array}{l}
g\rightharpoonup g = g\leftharpoonup g = g,\\
g\rightharpoonup h = h\leftharpoonup g = h,\\
h\rightharpoonup g = g\leftharpoonup h = g,\\
g\rightharpoonup x = h\rightharpoonup x = x\leftharpoonup g = x\leftharpoonup h =-x,\\
x\rightharpoonup g = x\rightharpoonup h =g\leftharpoonup x = h\leftharpoonup x =0,\\
x\rightharpoonup x = x\leftharpoonup x = \alpha(1-g),
\end{array}\right.$$
and then as $h\rightharpoonup h \neq h\rightharpoonup g=g$, we have the following two situations to consider:
\begin{enumerate}[(i)]
\item
$h\rightharpoonup h = h\leftharpoonup h = h$,
\item
$h\rightharpoonup h = h\leftharpoonup h = gh$.
\end{enumerate}

Correspondingly, for all other interactions involving basis elements $gh,gx,hx,ghx$, one can calculate them by Eqs.~\eqref{eq:MP1}--\eqref{eq:MP4}. For example,
\begin{eqnarray*}
gx\rightharpoonup gx &=& g\rightharpoonup (x\rightharpoonup gx)\\
&\stackrel{\eqref{eq:MP1}}{=}& g\rightharpoonup ((x\rightharpoonup g)((1\leftharpoonup g)\rightharpoonup x)+(g\rightharpoonup g)((x\leftharpoonup g)\rightharpoonup x))\\
&=& g\rightharpoonup (-g(x \rightharpoonup x))\\
&=& -\alpha(g\rightharpoonup g(1-g))\\
&=& \alpha(1-g).
\end{eqnarray*}

After checking that both $\rightharpoonup$, $\leftharpoonup$ derived are really module coalgebra actions, we finally obtain by Theorem~\ref{thm:construct-mp} the following two families of matched pairs of actions on $A_{C_2\times C_2}$ parameterized by $\alpha\in \bk$, both of which are extended from the unique one on $H_4$.
Since the right action $\leftharpoonup$ can be obtained by Eq.~\eqref{eq:r-action} through the left one, we only present the multiplication tables of the left action $\rightharpoonup$ for simplicity.

\begin{table}[h]
\begin{center}
{\small \caption{ The 1st family of matched pairs of actions on $A_{C_2\times C_2}$}\label{tb:1}
\renewcommand{\arraystretch}{1.1}
{\begin{tabular}{|c|c|c|c|c|c|c|c|c|}
\hline
$\rightharpoonup$ & $1$ & $g$ & $h$ & $gh$ & $x$ & $gx$ & $hx$ &$ghx$\\
\hline
$1$ & $1$ & $g$ & $h$ & $gh$ & $x$ & $gx$ & $hx$ &$ghx$ \\
\hline
$g$ & $1$ & $g$ & $h$ & $gh$ & $-x$ & $-gx$ & $-hx$ &$-ghx$\\
\hline
$h$ & $1$ & $g$ & $h$ & $gh$ & $-x$ & $-gx$ & $-hx$ &$-ghx$\\
\hline
$gh$ & $1$ & $g$ & $h$ & $gh$ & $x$ & $gx$ & $hx$ &$ghx$ \\
\hline
$x$ & $0$ & $0$ & $0$ & $0$ & $\alpha(1-g)$ & $\alpha(1-g)$ & $\alpha(gh-h)$ &$\alpha(gh-h)$\\
\hline
$gx$ & $0$ & $0$ & $0$ & $0$ & $\alpha(1-g)$ & $\alpha(1-g)$ & $\alpha(gh-h)$ &$\alpha(gh-h)$\\
\hline
$hx$ & $0$ & $0$ & $0$ & $0$ & $\alpha(1-g)$ & $\alpha(1-g)$ & $\alpha(gh-h)$ &$\alpha(gh-h)$\\
\hline
$ghx$ & $0$ & $0$ & $0$ & $0$ & $\alpha(1-g)$ & $\alpha(1-g)$ & $\alpha(gh-h)$ &$\alpha(gh-h)$\\
\hline
\end{tabular}}

\vspace{1.5em}
\caption{ The 2nd family of matched pairs of actions on $A_{C_2\times C_2}$}\label{tb:2}
\renewcommand{\arraystretch}{1.1}
{\begin{tabular}{|c|c|c|c|c|c|c|c|c|}
\hline
$\rightharpoonup$ & $1$ & $g$ & $h$ & $gh$ & $x$ & $gx$ & $hx$ &$ghx$\\
\hline
$1$ & $1$ & $g$ & $h$ & $gh$ & $x$ & $gx$ & $hx$ &$ghx$ \\
\hline
$g$ & $1$ & $g$ & $h$ & $gh$ & $-x$ & $-gx$ & $-hx$ &$-ghx$\\
\hline
$h$ & $1$ & $g$ & $gh$ & $h$ & $-x$ & $-gx$ & $ghx$ &$hx$\\
\hline
$gh$ & $1$ & $g$ & $gh$ & $h$ & $x$ & $gx$ & $-ghx$ &$-hx$ \\
\hline
$x$ & $0$ & $0$ & $0$ & $0$ & $\alpha(1-g)$ & $\alpha(1-g)$ & $\alpha(gh-h)$ &$\alpha(gh-h)$\\
\hline
$gx$ & $0$ & $0$ & $0$ & $0$ & $\alpha(1-g)$ & $\alpha(1-g)$ & $\alpha(gh-h)$ &$\alpha(gh-h)$\\
\hline
$hx$ & $0$ & $0$ & $0$ & $0$ & $\alpha(1-g)$ & $\alpha(1-g)$ & $\alpha(h-gh)$ &$\alpha(h-gh)$\\
\hline
$ghx$ & $0$ & $0$ & $0$ & $0$ & $\alpha(1-g)$ & $\alpha(1-g)$ & $\alpha(h-gh)$ &$\alpha(h-gh)$\\
\hline
\end{tabular}}}
\end{center}
\end{table}
\end{exam}

\begin{theorem}\label{thm:A22}
All the Yang-Baxter operators associated to matched pairs of actions on the Hopf algebra $A_{C_2\times C_2}$ are involutive.
\end{theorem}
\begin{proof}
The first family of matched pairs of actions on $A_{C_2\times C_2}$ given in TABLE~\ref{tb:1} can be induced by a class of cotriangular structures $\calr_\alpha$ with parameter $\alpha\in \bk$ (see Theorem~\ref{thm:coq-mpa}), where
$$\calr_\alpha(g^ih^jx^k\otimes g^lh^mx^n)=
\delta_{k,0}\delta_{n,0}(-1)^{(i+j)(l+m)}+\delta_{k,1}\delta_{n,1}\alpha(-1)^{(i+j)(l+m+1)},\quad i,j,k,l,m,n=0,1.$$
Since $\calr_\alpha^{-1}=\calr_\alpha^{\rm op}$, this family of matched pairs of actions produce involutive Yang-Baxter operators defined by \eqref{eq:cqt-braid}.

On the other hand, the second family of matched pairs of actions on $A_{C_2\times C_2}$ given in TABLE~\ref{tb:2} can not be derived from any coquasitriangular structures, as some interactions between group-like elements are non-trivial, e.g. $h\rightharpoonup h=h\leftharpoonup h=gh\neq h$.
In spite of this, by checking that Eq.~\eqref{eq:in3} hold for this family of matched pairs of actions, we see that their associated Yang-Baxter operators are also involutive according to Theorem~\ref{thm:involutive}.

Indeed, based on the data in TABLE~\ref{tb:2}, we have
\begin{eqnarray*}
g^ih^j \rightharpoonup g^kh^l &=& g^{k+jl}h^l,\\
g^ih^jx \rightharpoonup g^kh^l &=& 0,\\
g^ih^j \rightharpoonup g^kh^lx &=& (-1)^{i+j(l+1)}g^{k+jl}h^lx,\\
g^ih^jx \rightharpoonup g^kh^lx &=& (-1)^{(j+1)l}\alpha(1-g)h^l
\end{eqnarray*}
for any $i,j,k,l=0,1$. Take for example the LHS of Eq.~\eqref{eq:in3} to be
\begin{eqnarray*}
g^ih^jx \leftharpoonup g^kh^lx &\stackrel{\eqref{eq:r-action}}{=}&
S(g^ih^jx \rightharpoonup g^kh^lx)g^ih^jg^kh^l+
S(g^{i+1}h^j \rightharpoonup g^kh^lx)g^ih^jxg^kh^l\\
&&\quad+S(g^ih^jx \rightharpoonup g^{k+1}h^l)g^ih^jg^kh^lx
+S(g^{i+1}h^j \rightharpoonup g^{k+1}h^l)g^ih^jxg^kh^lx\\
&=& (-1)^{(j+1)l}\alpha(1-g)h^lg^ih^jg^kh^l
+(-1)^{i+j(l+1)}x g^{k+jl+1}h^l g^ih^jxg^kh^l\\
&=& (-1)^{i+k+(j+1)l}\alpha(1-g)h^j.
\end{eqnarray*}
Correspondingly, the RHS of Eq.~\eqref{eq:in3} equals to
\begin{eqnarray*}
&&S(g^ih^jx \rightharpoonup g^kh^lx)\rightharpoonup g^ih^j+
S(g^{i+1}h^j \rightharpoonup g^kh^lx)\rightharpoonup g^ih^jx\\
&=& (-1)^{(j+1)l}\alpha(1-g)h^l\rightharpoonup g^ih^j
+(-1)^{i+1+j(l+1)}x g^{k+jl+1}h^l\rightharpoonup g^ih^jx\\
&=& (-1)^{i+j+k+l} g^{k+jl+1}h^lx \rightharpoonup g^ih^jx\\
&=& (-1)^{i+k+(j+1)l}\alpha(1-g)h^j.
\end{eqnarray*}
So they coincide with each other.
\end{proof}

\begin{remark}
In a sequel paper~\cite{LX}, we further classified matched pairs of actions on the Kac-Paljutkin Hopf algebra $H_8$. It is the unique 8-dimensional non-commutative and non-cocommutative semisimple Hopf algebra. In comparison with $A_{C_2\times C_2}$, there are totally six matched pairs of actions on $H_8$, but only two of them provide involutive Yang-Baxter operators.
\end{remark}

\vspace{0.1cm}
 \noindent
{\bf Acknowledgements} \,This work is supported by National Natural Science Foundation of China (12071094, 12171155), and Basic and Applied Basic Research Foundation of Guangdong Province (2022A1515010357).

\bibliographystyle{amsplain}

\end{document}